\documentclass[11pt]{amsart}

\usepackage{amscd,amssymb,amsthm,verbatim}
\usepackage{enumerate}
\usepackage{bm}
\usepackage[margin=1.4in]{geometry}
\usepackage{mathrsfs, graphicx} 
\usepackage{stmaryrd}

\usepackage[%
bookmarks=true,
colorlinks,
linkcolor=blue,
urlcolor=blue,
citecolor=blue,
plainpages=false,
pdfpagelabels,
final,
breaklinks=true\between
]{hyperref}
\usepackage{hyperref}
\usepackage[numbers,sort&compress]{natbib} 
\renewcommand{\bf}{\textbf}

\numberwithin{equation}{section}

\theoremstyle{plain}
\newtheorem{thm}{Theorem}[section]
\newtheorem{prop}[thm]{Proposition}
\newtheorem{lem}[thm]{Lemma}
\newtheorem{cor}[thm]{Corollary}
\newtheorem{conj}[thm]{Conjecture}

\theoremstyle{definition}
\newtheorem{defn}[thm]{Definition}
\newtheorem{ex}[thm]{Example}\theoremstyle{remark}
\newtheorem{rk}[thm]{Remark}
\newcommand{\ve}{\varepsilon}

\newcommand{\inj}{\operatorname{inj}}

\newcommand{\loc}{\mathrm{loc}}
\newcommand{\spt}{\operatorname{spt}}

\newcommand{\csum}{\mathbin{\#}}

\newcommand{\g}{\operatorname{genus}}
\newcommand{\current}[1]{\llbracket #1\rrbracket}

\DeclareFontFamily{U}{MnSymbolC}{}
\DeclareSymbolFont{MnSyC}{U}{MnSymbolC}{m}{n}
\DeclareFontShape{U}{MnSymbolC}{m}{n}{
	<-6>  MnSymbolC5
	<6-7>  MnSymbolC6
	<7-8>  MnSymbolC7
	<8-9>  MnSymbolC8
	<9-10> MnSymbolC9
	<10-12> MnSymbolC10
	<12->   MnSymbolC12}{}
\DeclareMathSymbol{\intprod}{\mathbin}{MnSyC}{'270}

\DeclareRobustCommand{\aech}{\text{\reflectbox{$\intprod$}}}
\newcommand{\Mrest}{\mathbin\aech}

\begin{document}
\title[Positive Scalar Curvature on Noncompact Manifolds]{Positive Scalar Curvature on Noncompact Manifolds and the Liouville Theorem}
\author{Martin Lesourd, Ryan Unger, and Shing-Tung Yau}

\address{Black Hole Initiative, Harvard University, Cambridge, MA 02138}
\email{mlesourd@fas.harvard.edu}

\address{Department of Mathematics, Princeton University, Princeton, NJ 08544}
\email{runger@math.princeton.edu}

\address{Department of Mathematics, Harvard University, Cambridge, MA }
\maketitle

\begin{abstract}
Using minimal hypersurfaces, we obtain topological obstructions to admitting complete metrics with positive scalar curvature $R>0$ on a given class of non-compact manifolds $M^n$ for $n<8$. \\ \indent 
We show that the Liouville theorem for locally conformally flat manifolds $M^n$ with $R\geq 0$ follows from the impossibility of admitting a complete metric with $R>0$ on $T^n \# M^n$. With the recent work of Chodosh-Li, the Liouville theorem is now proved in all remaining cases. \\ \indent 
Finally, using MOTS instead of minimal hypersurfaces, we show an Initial Data Set version of these results with the Dominant Energy Scalar $\mu-|J|$ appearing instead of $R$. 
\end{abstract}

\section{Main results}
All manifolds are assumed smooth and orientable. A classic result of Schoen--Yau \cite{SY79} for $n\leq 7$ and Gromov--Lawson \cite{GL80, GL83} for all $n$ states that the $n$-torus $T^n$ does not admit a Riemannian metric with positive scalar curvature $R>0$. 
\begin{thm}[Geroch's Conjecture]
$T^n$ does not admit a metric with $R>0$.
\end{thm}
Interest in $T^n$ originally came from the Positive Mass Theorem \cite{SY79PMT} in General Relativity, which states that an asymptotically flat manifold with $R\geq 0$ has non-negative ADM mass and zero if and only if the manifold is flat.\footnote{See Chap. 3 of \cite{Lee} for the definitions of ADM mass and asymptotic flatness.} Although not originally approached in this way, a compactification argument of Lokhamp \cite{L99} showed that the Positive Mass Theorem follows from the non-existence of a smooth complete metric with $R>0$ on compact manifolds of the form $T^n \csum X^n$.  \\ \indent
Here, letting $X^n$ be non-compact and replacing $T^n$ with the class of $\mathcal C^n$ manifolds\footnote{Definition 1.2 is given in \cite{G18} and inspired by the proof in \cite{SY79}, though note that the labels $C_n$ and $C'_n$ in \cite{SY79} are related to but not identical to $\mathcal C^n$.}, we show the following.
\begin{thm}\label{thm:main}
Manifolds of the form $M^n=\mathcal C^n \csum X^n$ do not admit a complete smooth metric with $R>0$ if either
\begin{itemize}
\item[(i)] $n=3$,
\item[(ii)] or $4\leq n\leq 7$ and $M^n$ is $(\lambda,\rho)$-contractible.
\end{itemize}
and moreover any complete smooth metric with $R\geq 0$ is Ricci-flat.
\end{thm} 
\begin{defn}
A closed orientable manifold $Y^n$ is said to be $\mathcal C^n$ if it admits a smooth map $f:Y^n\to T^{n-2}$ such that the homology class of the preimage of a generic point,
\[[f^{-1}(t)]\in H_2(Y^n),
\]
is \textit{non-spherical}, i.e., it is not in the image of the Hurewicz homomorphism $\pi_2(Y)\to H_2(Y)$. 
\end{defn}
\begin{defn}
A manifold $M$ is $(\lambda,\rho)$-contractible if there exists constants $(\lambda,\rho)>0$ such that all balls of $M$ of radii $r\leq \rho$ are $\lambda$-Lipchitz contractible in $M$, i.e., there exist $\lambda$-Lipchitz maps 
\begin{equation*}
\phi=\phi_{x,r}:B_x(r)\times [0,r]\to M, \: x\in M,\: r\in [0,\rho]
\end{equation*}
where the maps $\phi(\cdot, 0)$ are the original embeddings $B_x(r)\subset M$ and the maps $\phi(\cdot, r)$ are constant.
\end{defn} 
We believe that $(\lambda,\rho)$-contractibility is technical, and it is likely that the restriction $n\leq 7$ can be lifted following \cite{SY17}.  \\ \indent 
To prove Theorem 1.3, we construct a sequence of minimizers $\{\Sigma_i\}$ and show convergence in $C^\infty_{\text{loc}}$ to a limiting hypersurface $\Sigma$. When $M^n$ is non-compact, such a sequence can degenerate to a non-compact hypersurface $\Sigma$, but for $n=3$ results about stable minimal hypersurfaces in the presence of $R>0$ are strong enough to yield a contradiction. For $4\leq n\leq 7$, $(\lambda,\rho)$-contractibility implies that $\Sigma$ is in fact compact, permitting for the argument in \cite{SY79} to go through.\footnote{When $\Sigma$ is non-compact, the stability argument of \cite{SY79} seems hampered by the possible lack of completeness of the conformal metric.} \\ \indent 
Recently and by different arguments, Chodosh--Li \cite{CL20} have shown that orientable manifolds of the form $M^n=\left(\mathcal{C}^{n-1}\times S^1\right) \csum X^n$ cannot admit a complete smooth metric with $R>0$ for $3\leq n\leq 7$. The $S^1$ factor is used to obtain prescribed mean curvature surfaces (``$\mu$-bubbles") from a carefully chosen functional, which, by the warped product method along with an associated stability inequality, are manipulated to yield a topological contradiction. The use of $\mu$-bubbles to study similar problems in scalar curvature also feature in \cite{G96, G18, G20, Z20}.  \\ \indent 
We also note that based on the ideas in \cite{SY17} and \cite{CS18}, it is readily shown that $Y^n\#X^n$ does not admit uniformly positive scalar curvature $R\geq \sigma>0$ for $n\leq 7$ if $Y^n$ is enlargeable, which indeed highlights the additional difficulty in dealing with $R>0$ instead of $R\geq \sigma>0$. \\ \indent
One of the motivations for Theorem \ref{thm:main} is the following Liouville Conjecture for locally conformally flat (LCF) manifolds $M^n$ with nonnegative scalar curvature $R\geq 0$. 
\begin{conj}[LCF Liouville Conjecture]\label{SYLC}
Let $(M^n,g)$, $n\geq 3$, be a complete, LCF, $R\geq 0$ manifold, and $\Phi:M^n\to S^n$ a conformal map. Then $\Phi$ is injective and $\partial\Phi(M)$ has zero Newtonian capacity.
\end{conj}
Schoen--Yau \cite{SY88} show this for $n\geq 7$, and, seperately, for $4\leq n\leq 6$ under various assumptions on the scalar and Ricci curvature. Here, we obtain the following improvement. 
\begin{thm}\label{SYL}
The nonexistence of a complete smooth metric with $R>0$ on $T^n \# M^n$ implies Conjecture \ref{SYLC}, and is thus settled for $n=3$, and for $4\leq n\leq 6$ if $M^n$ is $(\lambda,\rho)$-contractible.
\end{thm}
By Theorem \ref{SYL}, the aforementioned work of Chodosh--Li \cite{CL20} settles Conjecture \ref{SYLC} in the remaining cases. We note that in \cite{SY88}, it is shown that Conjecture \ref{SYLC} follows from the \textit{positive mass conjecture with arbitrary ends}. 
\begin{conj}[Positive Mass Conjecture with Arbitrary Ends]\label{PMC}
Let $N^n$ be a complete manifold diffeomorphic to $\mathbb{R}^n$ and let $X^n$ be a complete connected non-compact manifold. Let $g$ be a smooth complete metric on $M^n=N^n \csum X^n$ with nonnegative scalar curvature $R\geq 0$. If $g$ is asymptotically flat on $N^n$ with suitable fall-off, then the ADM mass of $(M^n,g)$ measured in the asymptotically flat end $N^n$ is nonnegative. 
\end{conj}
By ``suitable fall-off", one could seek to reproduce that of the standard positive mass theorem.\footnote{See chapter 3 of \cite{Lee}.} Lokhamp-type compactification arguments do not seem to allow reducing Conjecture \ref{PMC} to a statement about $T^n\csum X^n$, but it turns out that we do not need the full strength of Conjecture \ref{PMC} to prove Theorem \ref{SYL}. \\ \indent 
Finally, in the spirit of one of Gromov's questions \cite{G17} to look for other curvature conditions which behave like $R>0$, we show the following initial data set generalization of Theorem \ref{thm:main}.
\begin{defn}
An \emph{initial data set} is a triple $(M^n,g,k)$, where $g$ is a Riemannian metric on $M^n$, and $k$ a symmetric $2$-tensor on $M^n$. Associated to $(M^n,g,k)$ is a pair $(\mu,J)$, a scalar and a $1$-form, defined as
\[\mu\equiv R+\text{tr}_g(k)^2-|k|^2_g\:\:, \:\: J\equiv \text{div}_g(k-\text{tr}_g(k)).\]
\end{defn}
\begin{defn}
Let $X^n$ be a connected $n$-manifold without boundary (compact or not) and $(Y^n \#X^n,g,k)$ an initial data set. Let $\mathscr S$ denote a surgery sphere for the connected sum. Let $U(\mathscr S)=U_+(\mathscr S)\cup U_-(\mathscr S)$ be a collar neighborhood of $\mathscr S$, with $U_+(\mathscr S)$ pointing towards $Y^n$. We say that $(Y^n \#X^n,g,k)$ is an $\mu-|J|\geq0$ \emph{island} if $\mu-|J|\geq0(\neq 0)$ in $X^n\cup U(\mathscr S)$ and $Y^n\backslash U_+(\mathscr S)$ is flat.
\end{defn}
\begin{thm}\label{Loh}
$(T^3\#X^3,g,k)$ does not support $\mu-|J|\geq 0$-islands. 
\end{thm}
The proof of Theorem \ref{Loh} uses a deformation argument of Lokhamp \cite{L16} and is modelled on that of Theorem \ref{thm:main} except that marginally outer trapped surfaces (MOTS) take the role of minimal hypersurfaces.\\ \\
\textbf{Acknowledgements} We thank O. Chodosh and C. Li for discussions concerning their work. We thank C. Sormani for organizing the 2020 Virtual Workshop on Scalar and Ricci curvature, during which this work was presented and benefited from questions of the participants. M. Lesourd acknowledges the Gordon and Betty Moore Foundation and Harvard's Black Hole Initiative, R. Unger acknowledges Thomas Massoni and Maggie Miller, and S-T. Yau acknowledges the the support of NSF Grant DMS-1607871.

\section{Proof of Theorem \ref{thm:main}}

\subsection{Warm-up.} Recall the following very classical result. 

\begin{prop}
The torus $T^2$ does not admit a metric of positive scalar curvature.
\end{prop}
This can easily be proved using the Gauss--Bonnet theorem. However, following Synge and Cartan, we can argue as follows. The torus contains a nontrivial free homotopy class of curves $\mathcal L$. Using the Arzela--Ascoli theorem and basic results on geodesics, we may find a closed geodesic $\gamma\in\mathcal L$ of minimal length. The second variation formula implies this geodesic is unstable because of positive curvature, which is a contradiction. This technique motivated the minimal hypersurface approach to scalar curvature. 

The compactness of the torus is used in two places. Firstly, to apply the Arzela--Ascoli theorem to maps $S^1\to T^2$ and secondly, to ensure the positivity of the convexity radius of $(T^2,g)$. However, the Cartan--Synge argument goes through if we know that the ``relevant" curves all lie in some large ball. 

\begin{thm}\label{thm:punctured_torus} Let $X^2$ be any surface without boundary. Then
$T^2\# X^2$ does not admit a complete metric of positive scalar curvature. 
\end{thm}
\begin{proof}
Consider a free homotopy class $\mathcal L$ corresponding to going once around one of the circular factors in $T^2$. Let $C\subset T^2\# X^2$ be a ``surgery circle" associated to the connected sum. Then $C\notin \mathcal L$. In $T^2$, $C$ bounds an open disk. Call the complement of this disk $K\subset X^2\csum T^2$.


Let $\{\gamma_i\}\subset \mathcal L$ be a length-minimizing sequence of smooth curves. For each $i$, $\gamma_i\cap K\ne\emptyset$ for otherwise $\gamma_i$ would be nullhomotopic in $T^2$, contradicting the choice of $\mathcal L$. Let $\ell=\sup_iL(\gamma_i)<\infty$. Then by the triangle inequality, $\gamma_i$ is eventually contained in the $2\ell$-neighborhood of $K$. Since $g$ is complete, this neighborhood is relatively compact. 

Now a subsequence of $\{\gamma_i\}$ tends to a closed length-minimizing geodesic which can be shown to be unstable, and we have a contradiction.
\end{proof}

The topology of $X^2\csum T^2$ prevents the minimizing sequence from ``escaping off to infinity" -- the compact set $K$ \emph{anchors} $\mathcal L$. In fact, the completeness of $g$ prevented the curves from entering the noncompact end at all. This is much stronger than what can be expected for minimal hypersurfaces and in general we will have degeneration of topology at infinity.

\subsection{Homological anchoring}

In this section we set up the inductive machinery to prove Theorem 1.3. To use techniques of geometric measure theory, we will need to interpret homology groups in terms of flat chains.
We use the notation and terminology  of \cite{Federer}. Let $(A,B)$ be a pair of Lipschitz neighborhood retracts in $\Bbb R^L$ with $B\subset A$. Letting $\mathcal F_k$ denote the group of integral flat chains in $\Bbb R^L$ we define the \emph{integral cycles}
\[\mathcal Z_k(A,B)=\{T\in \mathcal F_k:\spt T\subset A, \spt\partial T\subset B\},\]
the \emph{integral boundaries}
\[\mathcal B_k(A,B)=\{T+\partial S\in \mathcal F_k+\partial\mathcal F_{k+1}:\spt T\subset B,\spt S\subset A\},\]
and the \emph{integral homology groups}
\[H_k(A,B)=\mathcal Z_k(A,B)/\mathcal B_k(A,B).\]

When $B=\emptyset$ we just write $\mathcal Z_k(A)$, etc. By Nash's theorem any Riemannian manifold is a Lipschitz retract. 

Note that the Mayer--Vietoris sequence holds for currents because current homology satisfies the Eilenberg--Steenrod axioms (see \cite[4.4.1]{Federer} and \cite[p.\ 161-162]{Hatcher}). 

\begin{defn}
Let $K\subset A$ be a compact set. A class $a\in H_k(A,B)$ is said to be \emph{homologically anchored} in $K$ if it can be represented by a chain with support in $K$ and cannot be represented by a chain whose support is disjoint from $K$. (Note that a flat chain always has compact support.)
\end{defn}

\begin{ex}
The homologically nontrivial ``central sphere" in the cylinder $S^{n-1}\times\Bbb R$ is not anchored in any set since it can be pushed off to infinity by an isotopy.
\end{ex}

\begin{ex} Let $X^n$ and $Y^n$ be smooth manifolds. To construct the connected sum $M=X\csum Y$ we remove a ball from $X$ and a ball from $Y$ and glue along the (now common) $(n-1)$-sphere $\mathscr S$. We let $\hat X$ denote the part of $X$ in $M$ together with a tubular neighborhood of $\mathscr S$ and similarly for $\hat Y$. 
\end{ex}

\begin{prop}\label{prop:csum2}
	Let $Y^n$ be a closed, orientable manifold and $X^n$ any orientable manifold. Then for $1\le i\le n-1$ there is an isomorphism of homology groups
	\begin{equation}H_i(X)\oplus H_i(Y)\cong H_i(M),\label{eqn:isom}\end{equation}
	where $M=X\csum Y$. 
\end{prop}

This is a classic exercise in algebraic topology but for geometric reasons we are interested in the maps involved.


\begin{cor}\label{cor:csum3}
	Suppose $a\in H_i(M)$ comes from a class in $H_i(Y)$ under the isomorphism \eqref{eqn:isom} and can be written as a sum of integral \textbf{cycles} in $M$, i.e. 
	\[A\sim A_1+A_2,\quad A_1,A_2\in  \mathcal{Z}_i(M),\]
	with the property that $\spt A_1\subset \hat X$ and $\spt A_2\subset \hat Y$. Then $A_2$ is a boundary in $M$. 
\end{cor}

When $A_2=0$ we obtain that homology classes coming from $Y$ are anchored in $\hat Y$.

\begin{proof}[Proof of Proposition \ref{prop:csum2}]
	Since $\hat X\cap \hat Y$ deformation retracts onto the central $(n-1)$-sphere $\mathscr S$, the reduced Mayer--Vietoris sequence for $\hat X $ and $\hat Y$ gives 
	\[0\to H_i(\hat X)\oplus H_i(\hat Y)\stackrel{\sim}{\to} H_i(M)\to 0\]
	for $1\le i\le n-2$. Using the long exact sequences of the pairs $(X,\hat X)$ and $(Y,\hat Y)$, one can easily see that there are isomorphisms $H_i(\hat X)\stackrel{\sim}{\to} H_i(X)$ and $H_i(\hat Y)\stackrel{\sim}{\to} H_i(Y)$, again for $1\le i\le n-2$, induced by the inclusions. 
	
	We now look at the case when $i=n-1$ and $X$ is compact. Then the relevant piece of the Mayer--Vietoris sequence reads 
	\[0\to H_n(\hat X)\oplus H_n(\hat Y)\to H_n(M)\stackrel{\delta}{\to}H_{n-1}(\mathscr S)\to H_{n-1}(\hat X)\oplus H_{n-1}(\hat Y)\to H_{n-1}(M)\to 0.\]  
	Since $\hat X$ and $\hat Y$ are noncompact and $M$ is compact, it can be simplified to 
	\[0\to \Bbb Z\stackrel{\delta}{\to} \Bbb Z\to H_{n-1}(\hat X)\oplus H_{n-1}(\hat Y)\to H_{n-1}(M)\to 0.\] 
	The connecting homomorphism $\delta:\Bbb Z\to\Bbb Z$ is injective by exactness so we just have to check it is surjective. But it is easy to realize $\current{\mathscr S}$ as the common boundary (up to sign) of a chain in $\hat X$ and a chain in $\hat Y$ whose sum is $ \current{M} $.  Next we show that $H_{n-1}(\hat X)\to H_{n-1}(X)$ is an isomorphism. The long exact sequence of the pair $(X,\hat X)$ reads (after making the proper evaluations)
	\[0\to H_n(X)\to H_n(X,\hat X)\to H_{n-1}(\hat X)\to H_{n-1}(X)\to 0.\]
	But since $X$ is compact orientable and $\hat X$ is homotopy equivalent to $X\setminus\{*\}$, $H_n(X)\to H_n(X,\hat X)$ is an isomorphism \cite[Theorem 3.26 (a)]{Hatcher}. It follows from exactness that $H_{n-1}(\hat X)\stackrel{\sim}{\to}H_{n-1}(X)$. The same argument works for $Y$.
	
	Finally, we consider the case $i=n-1$ and $X$ is noncompact. Now the Mayer--Vietoris sequence simply reads 
	\begin{equation}0\to H_{n-1}(\mathscr S)\to H_{n-1}(\hat X)\oplus H_{n-1}(\hat Y)\to H_{n-1}(M)\to 0\label{eqn:MV2}\end{equation}
	since $M$ is noncompact. The map $ H_{n-1}(\mathscr S)\to H_{n-1}(\hat X)\oplus H_{n-1}(\hat Y)$ is injective because $\mathscr S$ is not a boundary in $\hat X$, but will of course be a boundary in $H_{n-1}(M)$. Now $H_{n-1}(\hat Y)\to H_{n-1}(Y)$ is still an isomorphism but $H_{n-1}(\hat X)\to H_{n-1}(X)$ is merely surjective
	 (for example, $\Bbb R^n\setminus\{0\}\subset \Bbb R^n$). Look at the Mayer--Vietoris sequence for the covering of $X$ by $\hat X$ and the $n$-ball removed: 
	 \[0\to H_{n-1}(\mathscr S)\to H_{n-1}(\hat X)\to H_{n-1}(X)\to 0,\]
	 where we noted that $H_{n-1}(B^n)=0$. So we have 
	 \[H_{n-1}(X)\cong H_{n-1}(\hat X)/H_{n-1}(\mathscr S)\]
	 and $(*)$ gives 
	 \begin{equation}H_{n-1}(M)\cong (H_{n-1}(\hat X)/H_{n-1}(\mathscr S))\oplus H_{n-1}(\hat Y),\label{eqn:savior}\end{equation}
	where the injections $ H_{n-1}(\mathscr S)\to H_{n-1}(\hat X)$ are the same in both cases. By combining these isomorphisms, we obtain the claim. 
\end{proof}

\begin{proof}[Proof of Corollary \ref{cor:csum3}]
	The class $[A-A_1-A_2]_M=0$ is the image of $([A_1]_X,[A-A_2]_Y)$ under the isomorphism \eqref{eqn:isom}. So $[A]_Y=[A_2]_Y$ and $A_1$ is a boundary in $X$. If $X$ is compact or $1\le i\le n-2$ then $A_1$ is a boundary in $\hat X$ and hence also in $M$. If $X$ is noncompact and $i=n-1$ then
	\[[A_1]_{\hat X}\in \mathrm{image}\left(H_{n-1}(\mathscr S)\to H_{n-1}(\hat X)\right).\]
	But then by \eqref{eqn:savior}, $[A_1]_M=0$.  
\end{proof}


Using our simple homological argument, we can show that the $\mathcal{C}^n$ condition is in a certain sense ``contagious" with respect to connected sums. Note below that we will use $\mathcal{C}'^n$ to denote a class of manifolds having the property that a homology class in $H_2(M^n,\Bbb Z)$ cannot be represented by a sum of embedded $2$-spheres. 




\begin{defn}
A smooth manifold $M^n$ is said to be in class $\mathcal C'^{n}$ if it contains anchored classes $a_1,\dotsc, a_{n-2}\in H_{n-1}(M)$ whose intersection in $H_2(M)$ cannot be represented by a sum of multiples of embedded spheres. 
\end{defn}

\begin{rk} If $M^n$ is compact, then 
$\mathcal{C}^n$ is equivalent to $\mathcal{C}'^n$ for $n=3$ and $n\ge 5$. Indeed, $\mathcal C^n\implies \mathcal C'^n$ in all dimensions since a map $f:M\to T^{n-2}$ is given by $n-2$ maps $M\to S^1$, each of which give rise to $H_{n-1}$ classes by duality. Embedded spheres can be tubed together to give a map from a single sphere lying in the same homology class. Higher multiplicity spheres can also be represented by branched covers over a map from a single sphere.  

We can also see that $\mathcal C'^n\implies \mathcal C^n$ when $n\ne 4$. When $n\ge 5$, this is because a map from $S^2$ can be perturbed to be an embedding by transversality. For $n=3$, we use the fact that $\pi_2(M)$ is generated by the embedded spheres in the prime decomposition as a $\pi_1$-module \cite[Proposition 3.12]{Hatcher3M}. It follows that homology classes in the image of the Hurewicz map are represented by sums of multiples of these spheres.  

However, since the sphere theorem fails for $4$-manifolds (desingularizing a double point adds genus), it is not obvious to the present authors that $\mathcal C'^4\implies \mathcal C^4$. 
\end{rk}

\begin{prop}\label{contagious} Let $Y^n$ be a compact $\mathcal C^n$ manifold and set $M^n=Y^n \csum X^n$. Then $M^n$ is of class $\mathcal C'^n$.




\end{prop}

\begin{proof}
By the above remark, we obtain $b_1,\dotsc, b_{n-2}\in H_{n-1}(Y)$ whose intersection does not lie in the image of the Hurewicz map $\pi_2(Y)\to H_2(Y)$. We realize these in $M$ as anchored classes $a_1,\dotsc, a_{n-2}$ using the isomorphism \eqref{eqn:isom}. We can find mutually transverse hypersurfaces $\Sigma_1,\dotsc,\Sigma_{n-2}\subset \hat Y$ representing $b_i$, and hence $a_i$. We may also assume these are all transverse to the central sphere $\mathscr S$. 

Suppose $\Sigma_1\cap\cdots\cap \Sigma_{n-2}$ is homologous to a collection of embedded 2-spheres in $M$. If such a sphere $S$ lies in $\hat X$, it is a boundary by Corollary \ref{cor:csum3}. If $S$ intersects $\mathscr S$, we may assume it intersects in a collection of circles. Each circle can be capped by a map from $D^2$ lying in $\mathscr S$. By adding and subtracting this map from the homology class  $[S]$ we can push the two resulting spheres into $\hat Y$ and $\hat X$.\footnote{When $n=4$, $S\cap\mathscr S$ can be knotted, so the map from $D^2$ cannot be assumed to be an embedding.}

So $\Sigma_1\cap\cdots\cap \Sigma_{n-2}$ is homologous to a collection of spheres lying entirely in $\hat Y$. Using Proposition \ref{prop:csum2}, we can realize this as a representation of $b_1\bullet \cdots\bullet b_{n-2}$ by these spheres inside of $Y$. But these spheres can be tubed together to give an element of the image of $\pi_2(Y)\to H_2(Y)$. This contradicts the fact that $Y$ is of class $\mathcal C^n$. \end{proof}

\subsection{The Sequence}

Since we will be dealing with minimal hypersurfaces represending abstract homology classes, we need to make precise what our smooth objects are and what smooth convergence means. 

Let $(M^n,g)$ be a complete Riemannian manifold and $N_1,N_2,\dotsc$ be a sequence of smoothly embedded, $k$-dimensional, oriented submanifolds of $M$. If $N$ is another such submanifold, we say that $N_i\to N$ in $C^\infty_\mathrm{loc}$ if for every relatively compact domain $\Omega\subset M$, $N_i\cap\overline\Omega\to N\cap\overline\Omega$ in Hausdorff distance and for $i$ sufficiently large there exist functions $f_i:N\to \Bbb R^{n-k}$ such that $N_i\cap\Omega$ is the normal graph of $f_i$ over $N$ in $\Omega$ and $\|f_i\|_{C^k(N\cap\Omega)}\to 0$ for every $k\in\Bbb N$. It follows that in any bounded domain, $N$ and $N_i$ are eventually isotopic. 

We are interested in \emph{submanifolds with multiplicity}, i.e. $k$-dimensional integral currents whose regular set is their entire support. In other words, these are disjoint unions of the kinds of submanifolds considered above with integer multiplicities attached to each component. The convergence ``$N_i\to N$ in $C^\infty_\mathrm{loc}$" is defined componentwise and if the density (multiplicity) of some component of $m$, then we require $m$ sheets to converge to it at every point. 

In this context, we have the following compactness theorem for minimal hypersurfaces. 

\begin{thm}\label{thm:compactness}
	Let $M^n$ be a smooth manifold and $g_i\to g$ a sequence of complete Riemannian metrics converging in $C^\infty_\loc$. Let $\{\Sigma_i\}$ be a sequence of $g_i$-minimal hypersurfaces with multiplicity satisfying 
	\[\sup_i\left(\mathbf M_{g_i}(\Sigma_i\Mrest K)+\max_K|A_{\Sigma_i}|_{g_i}\right)<\infty\]
	for every compact set $K\subset M$. Then, after passing to a subsequence, $\{\Sigma_i\}$ converges in $C^\infty_\loc$ to a $g$-minimal hypersurface with multiplicity. 
\end{thm}

This notion of convergence allows for compact submanifolds to degenerate to noncompact ones. We have the following useful terminology.

\begin{defn}
Let $(M^n,g)$ be a complete Riemannian manifold and $[T]\in H_k(M)$. Let $N\subset M$ be a $k$-submanifold with multiplicity. We say that $N$ is \emph{weakly homologous} to $T$ if there exist submanifolds with multiplicity $N_1,N_2,\dotsc$, $N_i$ homologous to $T$, such that $N_i\to N$ in $C^\infty_\loc$.  We also say that $[T]$ \emph{degenerates} to $N$.  
\end{defn}

\subsection{Existence theorem}\label{ss:exhaust} We wish to apply standard techniques of GMT on specially modified compact subsets of our manifold. Let $(M^n,g)$ be a complete Riemannian manifold and $K\subset M$ a fixed compact set. It is standard that there exists a compact exhaustion $X_1\subset X_2\subset\cdots$ satisfying	
\begin{enumerate}[(i)]
	\item $\overline{X_i}\subset X_{i+1}$,
	\item $\partial X_i$ is a smoothly embedded submanifold, and 
	\item $K\subset X_1$.

	\item $\partial X_i$ is strictly outward mean convex with respect to a conformal metric $g_i$ on $X_i$, and
	
	\item $g=g_i$ outside a small neighborhood of $\partial X_i$ containing $X_{i-1}$.
\end{enumerate}
The metric $g_i$ on $X_i$ is constructed by a standard conformal deformation, cf. \cite{W19}, near $\partial X_i$ such that $g_i|_{X_{i-1}}=g|_{X_{i-1}}$. Let $h(t)$ be a positive smooth function on $\Bbb R$ with $h(t)=1$ for any $t\in \Bbb R\backslash [-\ve,\ve]$. Consider the function $f(x)\equiv h(d(x,\partial X_i))$ and the metric $g_{i}\equiv f^2 g|_{X_i}$. Then, under $(X_i,g_i)$, the mean curvature of $\hat{H}(x)$ of $\partial X_i$ is given by the well known formula \[ \hat{H}(x)=h^{-1}(0)\left(H(x)+2h'(0)h^{-1}(0)\right)\]
Thus it suffices to choose $\ve$ small enough and a function $h$ with $h(0)=2$ and $h'(0)>2\max_{x\in \partial X_i}|H(x)|+2$.\\ \\ 
We use this exhaustion to prove an existence theorem in anchored classes.

\begin{thm}\label{thm:exist}
	Let $(M^n,g)$ be a complete Riemannian manifold with $2\le n\le 7$. Let $a\ne 0\in H_{n-1}(M)$ be an anchored homology class. Then $a$ degenerates to a nontrivi{}al, properly embedded, locally area minimizing hypersurface with multiplicity. 
\end{thm}
\begin{proof}
	Let $K$ be the compact set which anchors $a$ and consider the compact exhaustion constructed above. Since $(X_i,g_i)$ is compact with strictly mean convex boundary, there exists a $g_i$-mass minimizing closed hypersurface with multiplicity $\Sigma_i\in a$, supported in $X_i$. We have 
	\[\mathbf  M_{g_i}(\Sigma_i)\le \mathbf  M_{g}(T),\]
	where $T$, $\spt T\subset K$, is a representative of $a$. By the Schoen--Simon curvature estimate \cite{SchoenSimon} we obtain a sequence of constants $C_1,C_2,\dotsc$ such that 
	\[\sup_i\left(\sup_{X_j}|A_{\Sigma_i}|\right)\le C_j.\]
	Using Theorem \ref{thm:compactness}, we obtain a subsequential limit of $\Sigma_i$ to a hypersurface with multiplicity $\Sigma$. 
	
	We now show that $\Sigma$ is nontrivial. Firstly, each $\Sigma_i$ intersects $X_1$ since $a$ is anchored there. We claim that there is a constant $\ve>0$ such that 
	\begin{equation}\mathbf M_g(\Sigma_i\Mrest X_1)\ge \ve\label{eqn:lb}
	\end{equation} 
	for every $i$.  Indeed, by the isoperimetric theorem \cite[4.4.2 (1)]{Federer} there is a number $\ve=\ve(X_1,g)>0$ such that if $\bf M_g(\Sigma_i\Mrest X_1)<\ve$, then 
	$\Sigma_i\Mrest X_1=T+\partial Y$, where $T\in \mathbf I_{n-1}(M)$ is supported in a tubular neighborhood of $\partial X_1$ and $Y\in \mathbf I_n(M)$. So if \eqref{eqn:lb} is not true, then by replacing $\Sigma_i$ in $X_1$ by $T$, we obtain a cycle representing $a$ whose support can be made disjoint from $X_1$ by flowing outward along the geodesic flow. This contradicts the anchoring of $a$. 
	
	Since mass is continuous under smooth convergence, \eqref{eqn:lb} implies $\Sigma$ is nontrivial. Since each $\Sigma_i$ is compact and the convergence is in $C^\infty_\loc$, $\Sigma$ is properly embedded. Finally, since $C^\infty_\loc$ convergence implies weak convergence of the underlying integral currents, $\Sigma$ is locally area minimizing by \cite[34.5]{Simon}.
\end{proof}

\subsection{Theorem \ref{thm:main}} Let $M^n=Y^n\# X^n$, where $Y^n$ is of class $\mathcal C'^n$. Let $a_1,\dotsc, a_{n-2}\in H_{n-1}(M)$ be the anchored homology classes given by Proposition \ref{contagious}. Minimize area in $a_1$ according to Theorem \ref{thm:exist} to obtain a locally area minimizing hypersurface with multiplicity $\Sigma$.


When $n=3$ we can apply the following classical result \cite{SY79, FCS80, SY82}.
\begin{thm}
	Let $(M^3,g)$ have positive scalar curvature and $\Sigma^2$ be a complete, stable, immersed minimal surface in $M$. If $\Sigma$ is compact, it is diffeomorphic to the sphere. If it is noncompact, it is conformally equivalent to $\Bbb C$ with its standard conformal structure. 
\end{thm}

\begin{rk}
If the scalar curvature is merely nonnegative, then a totally geodesic torus or punctured complex plane are allowed as well and there is a rigidity statement.
\end{rk}

Now no component of $\Sigma$ can be diffeomorphic to $\Bbb C$ since it has finite area. Indeed, Gromov and Lawson have shown that a stable minimal surface of finite area in a 3-manifold with $R>0$ is homeomorphic to a sphere \cite[Theorem 8.8]{GL83}.\footnote{A different argument for dealing with $\Bbb C$ not relying on the global area estimate is given in the proof of Theorem 1.10.} Therefore $\Sigma_i$ eventually consists of only spheres with multiplicity and perhaps other components lying in $\hat X$, which are homologically trivial by Corollary \ref{cor:csum3}. This contradicts the $\mathcal C'^3$ property. \\ \indent
As for $4\leq n\leq 7$, the two issues are that a characterization of stable minimal hypersurfaces in the presence of $R>0$ is no longer available, and moreover that owing to the presence of possibly non-compact components, the stability based argument of \cite{SY79} no longer yields that $\Sigma$ itself admits a \textit{complete} conformal metric with $R>0$. \\ \indent 
Given these difficulties, we assume that $M^n=\mathcal C^n \#X^n$ is itself $(\lambda,\rho)$-contractible. By arguments in \cite{G83} and \cite{G18}, this forces $\Sigma$ to be compact, which is sufficient for the induction argument \cite{SY79}. This is sketched in Step 3 of \cite[Sec. 11.6]{G18}. The key point is that $(\lambda,\rho)$-contractibilty allows one to use the localization trick used in the proof of \cite[Theorem $(\mathrm C_4)$ App. B]{G83} to generalize the arguments of 3.4.B and 3.4.C of \cite{G83} to non-compact manifolds. From there, one obtains the filling inequalities in Step 3 of \cite[Sec 11.6]{G18}. 
From there, Gromov claims a lower $(n-1)$-volume bound on the intersection between the minimizing hypersurface (minimal in our case) and the $n$-ball of radius of $\rho$. Compactness is then easily deduced. We explain this volume bound below. \indent
\begin{defn}
	Let $(M^n,g)$ be a complete Riemannian manifold.  Let $c,\rho>0$. We say that a submanifold $\Sigma^k\subset (M^n,g)$ is $(c,\rho)$-\emph{fat} if $\bf M(\Sigma\cap B_\rho(x))\ge c$ for every $x\in \Sigma$. 
\end{defn}
\begin{prop}\label{prop:fatcompact}
	If $\Sigma$ is a properly embedded connected submanifold which is $(c,\rho)$-fat and has finite mass, then $\Sigma$ is compact. In fact, if $x \in \Sigma$, then $\Sigma\subset B_{\rho\mathbf M(\Sigma)/c}(x)$. 
\end{prop}
\begin{proof}
If $\Sigma$ is noncompact, we can find distinct points $x_1,x_2,\dotsc$, tending to infinity, such that $B_\rho(x_i)\cap B_\rho(x_j)=\emptyset$ for $i\ne j$. This implies that $\Sigma$ has infinite mass, which is a contradiction. 
\end{proof}
\begin{defn}
	As in \cite{MSY}, we say that a complete manifold $(M^n,g)$ is \emph{homogeneously regular/bounded geometry} if its global injectivity radius is positive and all its sectional curvatures are uniformly bounded. 
\end{defn}
A short argument in \cite{MY} gives the following lower bound.
\begin{lem}
	Let $\Sigma$ be a minimal submanifold in a ball $B_\rho\subset M$, where $(M,g)$ is homogeneously regular and $0<\rho\le \inj(M,g)$. If $\partial \Sigma\subset \partial B_\rho$, then $\mathbf M(\Sigma)\ge c>0$, where $c$ depends only on $\rho$ and the ambient geometry. Hence $\Sigma$ is $(c,\rho)$-fat. 
\end{lem}

Since this implies compactness by Proposition \ref{prop:fatcompact}, we see that Theorem \ref{thm:main} (ii) holds under the assumption of homogeneous regularity: There is no homogeneously regular $R>0$ metric on a $\mathcal C^n\# X^n$ manifold. To explain Gromov's idea, we use a simple inequality.

\begin{lem}[Reverse Gronwall]
	Let $u:[0,T]\to\Bbb R_{\ge 0}$ be a continuous function, $u(t)>0$ for $t>0$. Suppose \[\int_0^t cu(s)^p\,ds\le u(t)\] 
	for some constants $c>0$, $0<p<1$, and all $t\in [0,T]$. Then 
	\[u(t)\ge c't^\frac{1}{1-p}\]
	for all $t\in[0,T]$. 
\end{lem}
\begin{proof}
	Let $v(t)=\int_0^t u^p$. Then $v'=u^p$ and one can follow the steps of the usual proof of the usual Gronwall's inequality.
\end{proof}
Let $\Sigma^k$ be an embedded $k$-dimensional submanifold and let $x\in \Sigma$. We set 
\[u(r)=\mathbf M(\Sigma\Mrest B_r(x)). \]
We set 
\[v(r)=\mathbf M(\Sigma\Mrest \partial B_r(x)).\]
The coarea formula implies
\[\int_0^r v \,ds\le u(r).\]
The idea is to fill $\Sigma\Mrest \partial B_r$ with a $k$-chain and use the isoperimetric inequality. Indeed, assuming $(\lambda,\rho)$-contractibility, we can find a $k$-chain $A$ with $\partial A=\Sigma\Mrest \partial B_r$ (at least for almost every $r$) and 
\[\mathbf M(A)\le c(\lambda,\rho,n) v(r)^{\frac{k}{k-1}}\]
If we assume $\Sigma$ is now area minimizing, we find 
\[u(r)\le c v(r)^\frac{k}{k-1}.\]
Substituting this into the coarea inequality,
\[\int_0^r cu^\frac{k-1}{k}\,ds\le u(r).\]
Using the reverse Gronwall inequality, we obtain the desired lower bound $u(r)\ge cr^{-k}$. We conclude:

\begin{prop}
Let $(M^n,g)$ be complete and $(\lambda,\rho)$-contractible. Then any properly embedded, area minimizing submanifold is $(c,\rho)$-fat, where $c=c(\lambda,\rho,n)$. 
\end{prop}

So again the minimizer associated to $a_1$ is compact and we may proceed with the induction, making use of the anchoring at each step. This completes the proof of the theorem.

\section{Proof of Theorem \ref{SYL}}

Let $(M,g)$ and $(N,h)$ be two Riemannian manifolds of the same dimension. A smooth map $F : M \to  N$ is said to be \textit{conformal} if there exists a smooth positive function $u$ on $M$ such that $F^* h=ug$. If $F$ is in addition a diffeomorphism, we say that it is a conformal diffeomorphism. If $F$ is a conformal map, $dF$ can never have a kernel, so $F$ is an immersion. \\ \indent 
An $n$-dimensional manifold $M$ is said to be \textit{locally conformally flat} (LCF) if it admits an atlas $A = \{(U_\alpha, \phi_\alpha)\}, \phi_{\alpha} : U_\alpha \subset M\to S^n$, such that whenever $U_{\alpha}\cap U_{\beta} \neq \emptyset$, the transition function
\[\phi_{\beta}\circ \phi^{-1}_{\alpha}:\phi_{\alpha}(U_{\alpha}\cap U_{\beta})\to \phi_\beta (U_{\alpha} \cap U_{\beta})\] 
is a conformal diffeomorphism of open subsets of $S^n$ with the standard round metric.

\subsection{Existence of the Conformal Green's Function}

Recall the conformal Green's function 
\[Lu=-\Delta u+aRu,\]
where $a=(n-2)/4(n-1)$. In this section we prove the existence of a unique minimal Green's function for the conformal Laplacian on certain open manifolds. 
\begin{prop}[]\label{thm:D1}
	Let $(M,g)$ be a locally conformally flat Riemannian manifold of dimension $n\ge 3$. Suppose there exists a conformal map $\Phi:M\to S^n$. For any $p\in M$, there is unique minimal and positive Green's function for the conformal Laplacian with pole $p$. It is $C^\infty$ away from $p$ and satisfies
	\[LG=\delta_p\]
	in the sense of distributions.
	\end{prop}

Here \emph{minimal} means that if $G'$ is any other positive Green's function for the conformal Laplacian with pole $p$, $G\le G'$ on $M\setminus\{p\}$. The proof of Proposition \ref{thm:D1} is outlined in \cite[Corollary 1.3]{SY88} and \cite[Proposition 2.4]{SYbook}. It is important to note that $L$ is a positive operator on a bounded LCF domain which immerses in the sphere \cite[Theorem 2.2]{SYbook}. 

\begin{proof}[Sketch of Proof]
	Let $\Omega_i\nearrow M$ be a compact exhaustion with $p\in \Omega_i$ a common point. On each domain $\Omega_i$ there is a unique Dirichlet Green's function $G_i$ for the conformal Laplacian with pole $p$. Near $p$, we have \[G_i(x)=\frac{c_n}{r^{n-2}}(1+o(1)),\] where $c_n^{-1}=n(n-1)|B^n_1|$.
	This is classical, but a careful proof can be found in \cite[Appendix A]{Hebey}. Since each $G_i$ has the same growth rate at $p$, the maximum principle shows $G_i\le G_{i+1}$ pointwise. We define $G$ to be the limit of this monotone sequence -- a priori this is not finite away from $p$. For this we need a barrier. 
	
	Since $\Phi:M\to S^n$ is conformal, $\Phi^*g_0=|d\Phi|^2g$, where $g_0$ is the round metric on $S^n$. Let $y=\Phi(p)$ and denote by $H$ the conformal Green's function of $(S^n,g_0)$ with pole $y$. One can check that 
	\[L\overline H=\sum_{q\in \Phi^{-1}(y)}|d\Phi(q)|^{-\frac{n-2}{2}}\delta_q\] in the sense of distributions,
	where $\overline H:=|d\Phi|^{\frac{n-2}{2}}H\circ\Phi$. Rescaling by an overall constant, we obtain a function $\overline G$ with poles located at every point in $\Phi^{-1}(y)$ and such that 
	\[L\overline G=\sum_{q\in \Phi^{-1}(y)}a_q\delta_q,\quad a_q>0,\quad a_p=1.\]
	
	Using $\overline G$ as a barrier, it can be shown that $G<\infty$ away from $\Phi^{-1}(y)$. That $G$ is bounded near each point in $\Phi^{-1}(y)\setminus\{p\}$ follows from the Harnack inequality applied to the sequence $\{G_i\}$. Since the sequence $\{G_i\}$ converges uniformly away from $p$, it is not hard to now show that $LG=\delta_p$. 
\end{proof}

\begin{rk}\label{rk:a}
	It follows from this construction that 
	\[G=\frac{c_n}{r^{n-2}}(1+o(1)). \]
	Indeed, $G$ is dominated by $\overline G$, which has this growth. 
\end{rk}

\subsection{Lohkamp compactification}

A significant simplification in the proof of the positive mass theorem was obtained by Lohkamp \cite{L99}, by reducing it to the torus rigidity theorem of Schoen--Yau and Gromov--Lawson. His argument was simplified further by Corvino--Pollack \cite{CP11} (see also \cite{Lee}). Here we observe the following additional fact.

\begin{prop}\label{Lohkamp}
	Let $g$ be a metric on $\Bbb R^n\setminus B$ with scalar curvature $R_g\ge 0$ and such that $g=u^{4/(n-2)}\delta$, where $u:\Bbb R^n\setminus B\to (0,\infty)$ is a smooth function satisfying 
	\[u(x)=1+cr^{2-n}+o(r^{2-n}),\]
	with $c<0$. Then the metric admits a Lohkamp compactification. 
\end{prop}
  By ``Lohkamp compactification" we mean that $g$ can be made totally flat outside of some large compact set without changing the sign on the scalar curvature, and without making it identically zero. By quotienting in some large cube, we compactify the end to a large flat torus component. 
  
  This proposition says harmonic approximation is not necessary -- superharmonic approximation suffices as long as we know that the conformal factor has a Kelvin-type asymptotic expansion.
  
\begin{proof}
	By the formula for scalar curvature under a conformal change,
	\[-\Delta_\delta u=c_nu^{\frac{n+2}{n-2}}R_g\ge 0,\]
	i.e. $u$ is superharmonic. We can find $\ve>0$ and $\rho_1,\rho_2>0$ so that $u(x)>1-\ve$ when $|x|>\rho_1$ and $u(x)<1-3\ve$ when $|x|<\rho_2$. The function $\min\{1-2\ve,u(x)\}$ is weakly superharmonic, so we utilize a smoothing of the minimum.  
	
	Explicitly, following Lee \cite{Lee}, define a $C^\infty$ function $\Phi:\Bbb R\to\Bbb R$ with the properties
	\[\Psi(t)=\begin{cases}
	t+2\ve \quad &\text{if } t<1-3\ve \\
	1 \quad &\text{if }  t>1-\ve \\
	\end{cases},\]
	\[\Psi'\ge 0,\text{ and, } \Psi''\le 0.\]
The observation that we may choose $\Psi$ nondecreasing is crucial here. Now define $\tilde u=\Psi(u)$. We compute
	\[\Delta \tilde u=\Psi''(u)|Du|^2+\Psi'(u)\Delta u\le 0,\]
	as desired. This is not identically zero because the first term is non-zero somewhere, i.e., there is a point where $Du\neq 0$ and $\Psi''(u)< 0$. Using $\tilde u$ as a conformal factor, we can carry out the Lohkamp compactification.
\end{proof}

\subsection{Finishing the proof} We will eventually use the Green's function $G$ as a conformal factor. In the following lemma we establish a precise decay rate.

\begin{lem}\label{3.2}
	The function $v=G/\overline G$ is a positive harmonic function wrt $\overline g=\overline G^\frac{4}{n-2}g$, smooth across $p$, and for any normal coordinate system $\{x^i\}$ centered at $p$, 
	\[v=1+cr^{n-2}+o(r^{n-2})\]
	for some constant $c$, where $r=|x|$. 
\end{lem}
\begin{proof}
	We outline the steps already contained in \cite{SY88, SYbook}. Let $\pi:S^n\setminus \{y\}\to \Bbb R^n$ be the stereographic projection. Since the Green's function $H$ on $(S^n,g_0)$ is actually just the conformal factor associated to stereographic projection, $\pi^*\delta=H^\frac{4}{n-2}g_0$, where $\delta$ is the flat metric on $\Bbb R^n$. It follows that $\Phi^*\pi^*\delta=\overline g$.

	It is easy to see that $v$ is harmonic with respect to $\overline g$ and also that 
	\[v(p)=\lim_{x\to p}\frac{G(x)}{\overline G(x)}=1,\]
	cf. Remark \ref{rk:a}. Setting $h=v-1$, $L(hG)=0$ since $hG=\overline G -G$. It follows from the removable singularities theorem for elliptic equations that $hG$ extends smoothly over $p$. For convenience we set $f=hG$. Then $h=fG^{-1}$. Since $f$ is $C^\infty$ and we have the expansion $G=c_nr^{2-n}+o(r^{2-n})$, it follows that $h=cr^{n-2}+o(r^{n-2})$.  
\end{proof}

\begin{proof}[Proof of Theorem \ref{SYL}]
	As explained in \cite{SY88, SYbook}, $\Phi$ is injective and $\partial \Phi(M)$ has zero Newtonian capacity if and only if $v\equiv 1$. Here we reduce this to the non-existence of $R>0$ on $T^n\csum X^n$ using the original idea of reducing it to the aforementioned positive mass conjecture with arbitrary ends.
	
	The conformal blowup $\tilde g=G^\frac{4}{n-2}g$ might not be complete since $G(x)\to 0$ as $d(x,p)\to\infty$ \cite[Lemma 3.2]{SY88}. To rectify this, we add a small parameter, 
	\[\tilde g_\ve=(G+\ve)^\frac{4}{n-2}g,\] is a complete metric on $\tilde M=M\setminus\{p\}$. 
	The formula for scalar curvature under a conformal deformation implies 
	\[R(\tilde g_\ve)=(G+\ve)^{-\frac{n+2}{n-2}}L_g(G+\ve)= (G+\ve)^{-\frac{n+2}{n-2}}\ve R(g)\ge 0.\]
	
	We now show that Proposition \ref{Lohkamp} is applicable. Since $\overline g$ is legitimately flat, the deviation of $\tilde g_\ve$ from flatness is determined entirely by $G$. Indeed in a neighborhood of $p$, 
	\[
	\tilde g_\ve	= \left(v+\frac{\ve}{\overline G}\right)^\frac{4}{n-2}\overline g= \Phi^*\pi^*\left(v_\ve^\frac{4}{n-2}\delta\right),\]
	where $v_\ve:\Bbb R^n\setminus B\to (0,\infty)$ satisfies
	\[v_\ve\circ \pi\circ\Phi = v+\frac{\ve}{\overline G}. \]
	By the expansion of $v$ shown in Lemma \ref{3.2} and the expansion $\overline G=c_nr^{2-n}+o(r^{2-n})$ it follows (after performing a coordinate inversion) that
	\[v_\ve =1+c_\ve r^{2-n}+o(r^{2-n})\] for a constant $c_\ve$. If $c_\ve<0$, we apply Proposition \ref{Lohkamp}, which is applicable to $v_\ve^\frac{4}{n-2}\delta$. The torus compactification can be pulled back to $(\tilde M,\tilde g_\ve)$ by $\pi\circ\Phi$. So we obtain a complete metric on a manifold $X^n$, diffeomorphic to the connected sum of $M^n$ and $T^n$, with nonnegative but not identically zero scalar curvature. By a theorem of Kazdan \cite{K82}, $X$ admits a complete PSC metric. This contradicts the conjecture that $T^n \csum M^n$ admits no $R>0$, so $c_\ve\ge 0$. 
	
	We can now let $\ve\to 0$ to conclude that the ADM mass of the blowup using $G$ itself is nonnegative. As shown in \cite{SY88, SYbook}, this suffices to show $v\equiv 1$. 
	
	When $n=3$, we are done, since $T^3\# M^3$ does not admit a complete PSC metric. 
	
	When $4\le n\le 6$, we show that if the original $(M,g)$ is $(\lambda,\rho)$-contracible, then the previous argument produces a $(\overline\lambda,\overline\rho)$-contractible PSC manifold with a torus component. Then we can use Theorem \ref{thm:main} (ii) to conclude our present theorem.
	
	Indeed, since the scalar curvature of $g$ is bounded below, $G(x)\to 0$ as $d(x,p)\to \infty$ \cite[Lemma 3.2]{SY88}. It follows that $G(x)\le C$ for some constant $C$ whenever $d(x,p)\ge 1$. So away from the asymptotically flat end, $(\tilde M,\tilde g_\ve)$ is bi-Lipschitz equivalent to  $(M,g)$. Additionally, the AF end of $(\tilde M,\tilde g_\ve)$ is homogeneously regular, so $(\tilde M,\tilde g_\ve)$ is globally $(\lambda',\rho')$-contractible for some $\lambda'$ and $\rho'$. 
	
	The next modification is the Lohkamp compactification procedure. However, since this step only modifies the AF end, we obtain a $(\lambda'',\rho'')$-contractible metric on $X$. Finally, Kazdan's deformation is explicitly a bi-Lipschitz equivalence, so we finally obtain a complete $(\lambda''',\rho''')$-contractible metric on $X$ with PSC.
\end{proof}

\section{Proof of Theorem \ref{Loh}}
The argument is modelled on that for Theorem \ref{thm:main} (i) except that several complications arise due to MOTS replacing minimal hypersurfaces. Recall that a closed two-sided surface $\Sigma$ immersed in $(M^n,g,k)$ is a \emph{marginally outer trapped surface} (\emph{MOTS}) if the null expansion $\theta^+|_{\Sigma}$ computed in the $l^+$ direction satisfies $\theta^+|_{\Sigma}=H+\text{tr}_{\Sigma}k=0$, where $l^+$ is a null normal vector field $l^+=n+\nu$ and $\nu$ its outward pointing normal in $M^n$. \\ \indent 
Let $\{X_i\}$ be a compact exhaustion of $T^3 \# X^3$, such that each $X_i$ contains $T^3 \cup U(\mathscr S)$. For each connected component of $\partial X_i$, fix a null normal vector field $l_i^+=\nu_i+n_i$ where $n_i$ is a future pointing timelike unit normal to the $X_i$, and $\nu_i$ is a spacelike unit normal for $\partial X_i$ which is pointing into $X_i$.\\ \indent
By a procedure similar to that in the proof of Theorem \ref{thm:main}, we can deform $g$ in a neighborhood of each connected component of $\partial X_i$ so that, in the new initial data, the null expansion on each component of $\partial X_i$ satisfies $\theta^+|_{\partial X_i}<0$, where $\theta^+|_{\partial X_i}$ is computed in the direction of $l_i^+$. \\ \indent 
One deforms $g$ to $f^{2}g$ in a neighborhood of $\partial X_i$. The deformation of $H$ was calculated previously. The mean curvature can be as negative as one wants by a reverse of the argument given above in the minimal surface case, i.e., pick $h'(0)$ negative. It remains to check that we can achieve $\hat{H}<-|\text{tr}_{\partial X_i}k|$ where $|\text{tr}_{\partial X_i}k|$ is now computed in the new initial data. This is clearly possible since $|\text{tr}_{\partial X_i}k|$ depends only on the absolute value of $g$ at $\partial X_i$ and not its derivative. Note that this deformation will in general destroy the positivity of $\mu-|J|$, but since that only occurs in a neighborhood of $\partial X_i$ which is sent to infinity, this causes no problems for the argument.\\ \indent
We now employ a result of Lokhamp in his work on the Spacetime Positive Mass Theorem in Sections 2 and 3 of \cite{L16}, which we briefly describe since these are not widely known. \\ \indent  
The first part of the argument involves re-distributing regions with positive $\mu-|J|$, cf. \cite[Corollary 2.4]{L16}. 
\begin{lem}\label{lem:L1}
Fix $(M,g,k)$. Suppose $\mu-|J|>0$ on some open $U\subset M$ and $\mu-|J|\geq 0$ outside. Let $W\subset M$ be an open with compact closure $\bar{W}\cap \bar{U}=\emptyset$ and let $(M,g',k')$ be an initial data set with $g'=g$, $k'=k$ outside $W$. Then, for any open connected $U^*$ such that $U\cup \bar{W}\subset U^*$, there is an $\varepsilon(U,U^*,\bar{W},g,h)>0$ so that if $\|g_1-g\|_{C^2}\leq \varepsilon$ and $\|k_1-k\|_{C^2}\leq \varepsilon$, there is a conformal transformation $v^{\frac{4}{n-2}}g$, $v^{\frac{2}{n-2}}k$ with $v=1$ outside $U^*$ such that $\mu-|J|>0$ in the new initial data in $U^*$. 
\end{lem}
This is shown by an explicit construction of a conformal factor, which itself is based on a specific cut-off function. After computing the change in $\mu-|J|$ under the conformal deformation $g\to v^{\frac{4}{n-2}}$, $k\to v^{\frac{2}{n-2}}k$
\begin{gather*}
\mu-|J|\to -\left(\gamma_n \Delta_g v+R_gv\right) v^{-\frac{n+2}{n-2}} \\
-\left( v^{-\frac{4}{n-2}}|k|_g^2-v^{-\frac{4}{n-2}}(\text{tr}_gk)^2\right)-|J(v^{\frac{4}{n-2}}g,v^{\frac{2}{n-2}}k)|_{v^{\frac{4}{n-2}}g}
\end{gather*}
where $\gamma_n=\frac{4(n-1)}{n-2}$ and $|J(v^{\frac{4}{n-2}}g,v^{\frac{2}{n-2}}k)|_{v^{\frac{4}{n-2}}g}=v^{-\frac{4}{n-2}}|J(g,k)|_g$, it is clear that the key estimate comes from the Laplacian term. After a straightforward estimate of this term, the result follows by constructing the right conformal factor. \\ \indent 
Lemma \ref{lem:L1} leads to the following, cf. \cite[Proposition 3.2]{L16}.
\begin{prop}\label{prop:L2}
For any given $\varepsilon,\gamma>0$, and $l\in \Bbb Z_{\geq 4}$, and any two disjoint parallel $T^{n-1}$ lying at $a_1,(\neq)a_2\in S^1$ in the flat portion of $\hat{T}^n$, there is a conformal deformation on $g_{n,\gamma}$ and $h_{n,\gamma}$ 
\[g_n(\varepsilon,\gamma,l|_{a_1,a_2}) = v^{\frac{4}{n-2}}_{\varepsilon,\gamma,l|_{a_1,a_2}} g_{n,\gamma} \: \: \: \: \emph{and} \: \: \: \: k(\varepsilon,\gamma,l|_{a_1,a_2})= v^{\frac{2}{n-2}}_{\varepsilon,\gamma,l|_{a_1,a_2}}  k_{n,\gamma}
\] 
such that 
\begin{enumerate}[{(i)}]
\item $\|g(\varepsilon,\gamma,l|_{a_1,a_2})-g_{n,\gamma}\|_{C^l(T^n\csum X^n)}\leq \varepsilon$,
\item $\mu-|J|$ is $>0$ in the new data on $T^n\csum X^n$, and
\item both tori $T^{n-2} \subset \partial \hat{Y}$ are strictly mean convex (away from the surgery sphere) in the in the new data.
\end{enumerate}
\end{prop}
The proof is also by explicit construction of the conformal factor using specific cut-off functions near the two $T^{n-1}$. A computation reveals that $\mu-|J|$ adopts a negative contribution from the cut-off function. One then uses the redistribution of Lemma \ref{lem:L1} to compensate. (iii) then follows by explicit computation of the second fundamental form of the two $T^2$ under the said conformal transformation. \\ \indent 
Since the arguments for Lemma \ref{lem:L1} and Proposition \ref{prop:L2} are purely local, we apply them to obtain two strictly mean convex $T^2$'s, $T^2_+$ and $T^2_-$. The outer null expansion $\theta^+|_{T^2_+}$ of $T^2_+$ (pointing away from $\mathscr S$) satisfies $\theta^+|_{T^2_+}>0$, the inner null expansion $\theta^-|_{T^2_-}$ (pointing towards $\mathscr S$) satisfies $\theta_-|_{T^2_-}<0$, and since from above $\theta^+|_{\partial X_i}<0$ on each component, we deduce from the results of Eichmair \cite{E09} the existence of a compact outermost stable MOTS $\Sigma_i$ embedded in $X_i$ that is homologous to the outer untrapped $T^2_+$.\footnote{In \cite{L16} the set-up being slightly different permits showing that such a MOTS is in fact quasi-isometric to $T^2\# N^2$ where $N^2$ is a closed orientable surface. This is perhaps morally true in our case but no such statement seems to be available.} \\ \indent
An important fact from \cite{E09} is that $\Sigma_i$ is `$C$-almost minimizing' in the sense of currents, where recall that a current $T$ is $C$-almost minimizing for some constant $C\ge 0$ in an open set $W$ if
\[\mathbf M_W(T)\le \mathbf M_W(T+\partial X)+C\mathbf M_W(X)\]
whenever $X$ is of one dimension higher than $T$ and $\spt X\subset W$. Say $T$ is the fundamental class carried by an orientable hypersurface $\Sigma$ intersecting $\partial W$ transversally. By choosing $X$ to be the fundamental class of a suitable collection of components of $W\setminus \Sigma$ ($X$ can be made to have compact support by an approximation argument), we can ensure that $\spt(T+\partial X)\subset \partial W$. In this case we have a mass bound
\[\mathbf M_W(T)\le \mathbf M(\partial W)+C\mathbf M(W)=C(W).\]
Eichmair's argument that his construction produces $C$-almost minimizing surfaces is local in the sense that, for a subset $\Omega'\subsetneq \Omega$, the constant $C_{\Omega'}$ depends only on $\Omega'$. This is because the MOTS is constructed by taking a limit of graphs of Jang functions in the cylinder over the data set. These Jang graphs have mean curvature bounded in $\Omega'$ by a constant depending on data in $\Omega'$. In Eichmair's notation, $G_t$ is $n\|k\|_{L^\infty(\Omega')}$-almost minimizing in $\Omega'\times\Bbb R$. The arguments in Eichmair, particularly Lemma A.1, will guarantee that the MOTS is $n\|k\|_{L^\infty(\Omega')}$-almost minimizing in $\Omega'$. \\ \indent
The second key fact is the curvature estimate for MOTS, shown independently by Andersson-Metzger \cite{AM10} and Eichmair \cite{E09} with slightly different hypotheses. In particular, Andersson--Metzger \cite{AM10} show 

\begin{lem}
Let $\Sigma^2$ be a symmetrized stable MOTS in a 3-dimensional compact initial data set $(M^3,g,k)$. Then $|A_\Sigma|\le C(g,k)$.
\end{lem}

They first prove a curvature bound in terms of area, and then a local area estimate for intrinsic balls in dimension 3. Alternatively, since the MOTS in our proof are constructed by Eichmair's method, one can use standard techniques in geometric measure theory to obtain a cuvature estimate for these MOTS.



We now consider the convergence of $\{\Sigma_i\}$. Since each $\Sigma_i$ is homologous to one of the outer untrapped $T^2$, a homological anchoring argument similar to that given in the proof of Theorem \ref{thm:main} keeps members of $\{\Sigma_i\}$ from leaving $U(\mathscr S)$.\\ \indent
At this point in the proof of Theorem \ref{thm:main}, we were able to show the convergence $C^{\infty}_{\text{loc}}$ to a properly embedded stable minimal hypersurface $\Sigma$. Unfortunately, even if each MOTS in the sequence $\{\Sigma_i\}$ is embedded, the limit $\Sigma$ can fail to be embedded. The reason for this is that a MOTS can ``touch itself" tangentially if the normals are misaligned at the point of tangency.  This phenomenon is well known even for constant mean curvature surfaces, cf. \cite[Remark 8.3]{AM10}. $\Sigma$ can thus fail to be embedded. Fortunately, the limiting MOTS turns out to be \textit{almost embedded}.
\begin{defn}
An immersed hypersurface $\Sigma$ is \emph{almost embedded} if, near each point in $M$ where the immersion map fails to be injective, $\Sigma$ can be written as finitely many sheets, each lying on one side of the other. 
\end{defn}
\begin{thm}\label{thm:MOTScompactness}
Let $M^n$ be a smooth manifold and $\{(g_i,k_i)\}$ a sequence of initial data, $g_i$ a complete metric and $k_i$ a symmetric $2$-tensor, converging in $C^{\infty}_{\text{loc}}$ to $(g,k)$. Let $\{\Sigma_i\}$ be a sequence of almost embedded $(g_i,k_i)$-MOTS with multiplicity satisfying 
\[
\sup_i \left( \mathbf{M}_{g_i} (\Sigma_i \Mrest K)+\max_K |A_{\Sigma_i}|_{g_i}\right) <\infty
\]
for every compact set $K\subset M$. Then, after passing to a subsequence, $\{\Sigma_i\}$ converges in $C^{\infty}_{\text{loc}}$ to a $(g,k)$-MOTS with multiplicity. Moreover, if $\Sigma_i$ is a boundary then $\Sigma$ is also composed of boundaries. 
\end{thm}
\begin{proof} Compare with the \cite[Theorem 2.11]{ZZ19}.
There is a one-sided maximum principle for MOTS, see Ashtekar--Galloway \cite{AG05}. So \cite[Lemma 2.7]{ZZ19} works for MOTS. From here, we can follow the proof of \cite[Theorem 2.11]{ZZ19} verbatim.
\end{proof}
Note that elsewhere in the non-compact ends of $X$, it could be that the number of components of $\Sigma_i$ grows with $i\to \infty$, which would ruin this property of $\Sigma$. Since we are only interested in what occurs within a compact set in $M$ between the two $T^2$ which also contains the surgery sphere, this is not an issue we have to deal with. \\ \indent
Moreover, it is easy to see that, since member of $\{\Sigma_i\}$ is stable, each leaf of $\Sigma$ is in fact symmetrized stable; that is,
\begin{equation}
\int_{\Sigma} |\nabla \phi|^2-\frac{1}{2}(|\chi|^2-R_{\Sigma}+2\tilde{\Psi}_M)\phi^2 d\mu \geq 0
\end{equation}
for all compactly supported $\phi$ on $\Sigma$, see \cite[Equation (27)]{AM10} for the definition of the associated curvature term $\tilde{\Psi}$. \\ \indent
We note in addition that the uniform area bound obtained above means that each leaf of the limit is proper. Indeed, if it is not proper, then there exist points $p_i\in \Sigma$ accumulating in $M$ but such that $B_r^\Sigma(p_i)$ are mutually disjoint. By the area lower bound obtained for surfaces with bounded mean curvature in \cite[Appendix B]{Fe96}, $M$ would have to have infinite area since each intrinsic ball contributes some definite amount of area, which is impossible.  \\ \indent
Given a non-compact component $\Sigma_{\mathrm{nc}}$ of $\Sigma$, symmetrized stability is sufficient to guarantee, from results in \cite{C14} which are themselves based on \cite{FCS80}, that $\Sigma_{\mathrm{nc}}$ is homeomorphic to either $\Bbb C$ or $\Bbb C^*$. The latter is forbidden by the rigidity that occurs in that case: $\mu-|J|$ must be $0$ along $\Sigma$. \\ \indent
A compact component $\Sigma_{\mathrm c}$ has finite area. In that case, a generalization of \cite[Theorem 8.8]{GL83} that replaces $R>0$ with $\mu-|J|>0$ and that holds for symmetrized stable immersed MOTS, yields that $\Sigma_{\mathrm c}$ has positive Euler characteristic and is thus homeomorphic to $S^2$ or ${\Bbb R} P^2$. We can rule out $\mathbb{R}P^2$ since the compactness theorem guarantees that the limits are boundaries, and thus orientable. See \cite[Section 4]{ALY20} for the details of this argument, which is stated for embedded MOTS but the immersed argument is the same. \\ \indent 
Considering now an anchoring set $K$ containing the surgery sphere. Within $K$, $\Sigma$ is a collection of finitely many almost embedded symmetrized stable MOTS, and are homeomorphic to portions of $\Bbb C$ and $S^2$. \\ \indent 
At this point in the proof of Theorem \ref{thm:main} we were able to use the global area estimate to rule out $\mathbb C$, but owing to the lack of such a bound in the current setting we employ a different argument.\\ \indent  
Let $\Sigma_i$ be the sequence of MOTS constructed above and suppose a component $S$ of the limiting MOTS $\Sigma$ is diffeomorphic to an almost embedded copy of $\Bbb C$.  Since $\Sigma$ is properly immersed, $\mathscr S\cap \Sigma$ is a finite disjoint union of immersed circles after possibly perturbing $\mathscr S$. Choose $i$ large enough that $\Sigma_i$ approximates $\Sigma$ well inside a large neighborhood of $\hat Y$. Now $\Sigma_i\cap \hat Y$ is abstractly diffeomorphic to $\Sigma\cap \hat Y$ (componentwise), hence
\[\g(\Sigma_i\cap\hat Y)=\g(\Sigma\cap \hat Y)=0 \quad\text{componentwise},\] and $\Sigma_i\cap \mathscr S$ is a union of (embedded) circles approximating $\Sigma\cap \mathscr S$. These circles bound (maps of) disks in $\mathscr S$, as $H^1(\mathscr S)=0$. We add and subtract these disks from $\Sigma_i$ and push them in opposite directions from $\mathscr S$. Abstractly, the cycles in $\hat Y$ created in this manner has genus $0$ since they are constructed by gluing disks to a genus $0$ surface. The cycles in $\hat X$ are boundaries by Corollary \ref{cor:csum3}. \\ \indent
$\Sigma$ is now a collection of embedded disjoint $S^2$, which contradicts the fact that each leaf is homologous to the outer untrapped $T^2_+$.

\bibliographystyle{acm}
\bibliography{SYSv3}

\end{document}